\documentclass[12pt,reqno]{amsart}

\newcommand\version{October 05, 2023}

\usepackage[T1]{fontenc}
\usepackage[utf8]{inputenc}
\usepackage{amsmath,amsfonts,amsthm,amssymb,amsxtra}
\usepackage{bbm} 
\usepackage[parfill]{parskip}    
\usepackage{enumerate}

\usepackage[ngerman, german, english]{babel} 

\newtheorem{theorem}{Theorem}[section]
\newtheorem{proposition}[theorem]{Proposition}
\newtheorem{lemma}[theorem]{Lemma}

\theoremstyle{definition}

\theoremstyle{remark}

\newtheorem{remark}[theorem]{Remark}

\numberwithin{equation}{section}

\renewcommand{\epsilon}{\varepsilon}

\newcommand{\R}{\mathbb{R}}

\newcommand*\diff{\mathop{}\!\mathrm{d}} 

\newcommand{\xl}{{x, \lambda}}

\newcommand{\Ds}{{(-\Delta)^{s}}}
\newcommand{\ds}{{(-\Delta)^{s/2}}}

\DeclareMathOperator{\dist}{dist}

\begin{document}

\title{Stability for the Sobolev inequality: existence of a minimizer}

\author{Tobias K\"onig}
\address[Tobias K\"onig]{Institut für Mathematik, 
Goethe-Universität Frankfurt, 
Robert-Mayer-Str. 10, 60325 Frankfurt am Main, Germany}
\email{koenig@mathematik.uni-frankfurt.de}

\date{\version}
\thanks{\copyright\, 2023 by the author. This paper may be reproduced, in its entirety, for non-commercial purposes.}

\maketitle

\begin{abstract}
We prove that the stability inequality associated to Sobolev's inequality and its set of optimizers $\mathcal M$ and given by
\[ \frac{\|\nabla f\|_{L^2(\mathbb R^d)}^2  - S_d \|f\|_{L^\frac{2d}{d-2}(\mathbb R^d)}^2}{ \inf_{h \in \mathcal M} \|\nabla (f -  h)\|_{L^2(\mathbb R^d)}^2 } \geq c_{BE} > 0 \qquad \text{ for every } f \in \dot{H}^1(\mathbb R^d),\]
which is due to Bianchi and Egnell, admits a minimizer for every $d \geq 3$. Our proof consists in an appropriate refinement of a classical strategy going back to Brezis and Lieb. As a crucial ingredient, we establish the strict inequality $c_{BE} < 2 - 2^\frac{d-2}{d}$, which means that a sequence of two asymptotically non-interacting bubbles cannot be minimizing. Our arguments cover in fact the analogous stability inequality for the fractional Sobolev inequality for arbitrary fractional exponent $s \in (0, d/2)$ and dimension $d \geq 2$. 
\end{abstract}

\section{Introduction and main results}

The famous stability inequality due to Bianchi and Egnell \cite{BiEg} states that, for dimension $d \geq 3$, there is a constant $c_{BE} > 0$ such that 
\begin{equation}
\label{bianchi-egnell s=1}
\frac{\lVert \nabla f\rVert  _{L^2(\R^d)}^2  - S_d \lVert f\rVert  _{L^\frac{2d}{d-2}(\R^d)}^2}{ \inf_{h \in \mathcal M} \lVert \nabla (f -  h)\rVert  _{L^2(\R^d)}^2 } \geq c_{BE}
 \qquad \text{ for all } f \in \dot{H}^1(\R^d)  \setminus \mathcal M.
\end{equation}
Here, 
\begin{equation}
\label{M definition}
\mathcal M := \left\{x \mapsto c (a + |x-b|^2)^{-\frac{d-2}{2}} \, : \, a > 0, \, b \in \R^d, \, c \in \R \setminus \{0\} \right\}
\end{equation}
is the $(d+2)$-dimensional manifold of \emph{Talenti bubbles}, i.e., optimizers of the Sobolev inequality 
\begin{equation}
\label{Sobolev s = 1}
\lVert \nabla f\rVert  _{L^2(\R^d)}^2 \geq S_d \lVert f\rVert  _{L^\frac{2d}{d-2}(\R^d)}^2
\end{equation} 
with sharp constant $S_d > 0$. Indeed, \eqref{bianchi-egnell s=1} makes a statement about the stability of the Sobolev inequality, in the sense that if the 'deficit' $\lVert \nabla f\rVert  _{L^2(\R^d)}^2  - S_d \lVert f\rVert  _{L^\frac{2d}{d-2}(\R^d)}^2$ is very small for some $f \in \dot{H}^1(\R^d)$, then $f$ must be very close to $\mathcal M$, in a quantitative fashion.

While the power two in the denominator of the left side of \eqref{bianchi-egnell s=1} is well known to be optimal, it is a long-standing open question to determine the value of the best constant $c_{BE}$ in \eqref{bianchi-egnell s=1}, see, e.g., the up-to-date surveys \cite{DoEs, Frank}. Indeed, the proof of \eqref{bianchi-egnell s=1} in \cite{BiEg} proceeds by compactness and therefore does not yield any positive lower bound on the value of $c_{BE}$.  Only very recently the first explicit lower bound on $c_{BE}$ has been established in the remarkable paper \cite{DoEsFiFrLo}. We also refer to \cite{BoDoNaSi, BrDoSi} for explicit constants in similar stability inequalities and to \cite{ChLuTa} for a related abstract result.

A key step towards determining the value of $c_{BE}$, which is explicitly mentioned as an open problem in \cite{DoEsFiFrLo}, consists in establishing the existence of an optimizer for \eqref{bianchi-egnell s=1}. To achieve this, one needs to investigate the behavior of a minimizing sequence $(f_n)$ for \eqref{bianchi-egnell s=1}. Must $(f_n)$ converge towards a minimizer, or,  on the contrary, is the optimal value of $c_{BE}$ only attained asymptotically along a certain sequence $(f_n)$ with zero weak limit in $\dot{H}^s(\R^d)$? 

We shall prove as a main result of this paper that the first alternative always holds. In particular, \emph{the Bianchi-Egnell inequality \eqref{bianchi-egnell s=1} always admits a minimizer}.

\subsection{Compactness vs. non-compactness of minimizing sequences. }
The fact that the existence of a minimizer for \eqref{bianchi-egnell s=1} cannot be proved in a straightforward way has to do with two natural scenarios for the behavior of minimizing sequences which could both prevent existence of a minimizer for $c_{BE}$. 

The simpler one of these scenarios consists of minimizing sequences $(f_n)$ that converge towards $\mathcal M$. (Since the quotient in \eqref{bianchi-egnell s=1} is ill-defined  for $f \in \mathcal M$, the limit of such sequences, even if non-zero, is not a minimizer.) The optimal value associated to this type of sequences can be obtained in terms of a spectral problem already analyzed in \cite{BiEg}, see also \cite{Rey}. It is given by $c_{BE}^\text{spec} := \frac{4}{d+4}$. In the recent article \cite{Koenig}, the author has excluded such behavior for minimizing sequences by showing the strict inequality $c_{BE} < c_{BE}^\text{spec}$. 

There is, however, another plausible scenario for non-compact minimizing sequences, namely a sequence $(u_n)$ consisting of two Talenti bubbles which are asymptotically non-interacting by virtue of having different concentration behavior and/or being centered far apart from each other. A back-of-the-envelope calculation, which can be made rigorous as in Section \ref{section strict inequality proof}, shows that for a configuration of two bubbles having equal mass and center, and diverging concentration rates, the quotient in \eqref{bianchi-egnell s=1} reaches the value $c_{BE}^\text{2-peak} := 2 - 2^\frac{d-2}{d}$ in the limit. (It can be checked that other model configurations built from non-interacting Talenti bubbles, including ones that involve more than two bubbles, do not yield a smaller value of the Bianchi-Egnell quotient.) A side remark one can make here is that, somewhat surprisingly, the question whether one or two bubbles yield a lower value in \eqref{bianchi-egnell s=1} turns out to depend on the dimension. Namely, $c_{BE}^\text{spec} < c_{BE}^\text{2-peak}$ for $3 \leq d \leq 6$, while $c_{BE}^\text{spec} > c_{BE}^\text{2-peak}$ for $d \geq 7$. 

Similarly to the strict inequality $c_{BE} < c_{BE}^\text{spec}$ from \cite{Koenig}, we will prove in Theorem \ref{theorem strict ineq 2 bubbles} that $c_{BE} < c_{BE}^\text{2-peak}$ strictly. Thus the two-bubble configurations described above cannot be minimizing sequences either. It turns out that the conjunction of these \emph{two} strict inequalities is enough to enforce the existence of a minimizer. This is the content of Theorem \ref{theorem minimizer} below.

Inequality \eqref{bianchi-egnell s=1} forms part of a wider class of geometric-type stability inequalities. Some celebrated quantitative stability results concern the Sobolev inequality for $\dot{W}^{1,p}$ \cite{CiFuMaPr, FiNe, FiZh}, the isoperimetric inequality \cite{FuMaPr}, the logarithmic Sobolev inequality \cite{BoGoRoSa, FaInLe} or the lowest eigenvalue of Schrödinger operators \cite{CaFrLi}. For most of these the existence of an optimal function is not known to date, let alone the explicit value of the stability constant corresponding to $c_{BE}$. A positive result one can mention here concerns the planar ($d=2$) isoperimetric inequality. For this case, an optimal set for the associated stability inequality is shown to exist in \cite{BiCrHe}. 

\subsection{Main results}

Since all of our arguments work identically for any fractional order $s \in (0, d/2)$, we shall state and prove our main results in this more general situation. That is, for any $d \geq 1$ and $s \in (0, d/2)$ we consider the fractional inequality of Bianchi-Egnell-type 
\begin{equation}
\label{bianchi-egnell}
 \inf_{f \in \dot{H}^s(\R^d) \setminus \mathcal M} \mathcal E(f) =: c_{BE}(s) > 0,
\end{equation}
for 
\begin{equation}
\label{E definition}
\mathcal E(f) := \frac{\lVert \ds f\rVert  _2^2 -S_d \lVert f\rVert  _{2^*}^2}{\dist(f, \mathcal M)^2}, 
\end{equation}
which was proved in \cite{ChFrWe}. Here and in the rest of the paper, we abbreviate $\lVert \cdot\rVert  _{L^p(\R^d)} = \lVert \cdot\rVert  _p$ and let 
\[ 2^* := \frac{2d}{d-2s} \]
be the critical Sobolev exponent. We denote by 
\[ \mathcal M := \left\{ x \mapsto c (a + |x-b|^2)^{-\frac{d-2s}{2}} \, : \, a > 0, \, b \in \R^d, \, c \in \R \setminus \{0\} \right \} \]
the manifold of fractional Talenti bubbles, i.e. optimizers of the Sobolev inequality 
\begin{equation}
\label{sobolev quotient definition}
\mathcal S(f) := \frac{\lVert \ds f\rVert  _2^2}{\lVert f\rVert  _{2^*}^2} \geq S_d > 0. 
\end{equation}
with sharp constant $S_d$. The homogeneous Sobolev space $\dot{H}^s(\R^d)$ is the completion of $C^\infty_0(\R^d)$ with respect to the norm $\lVert \ds f\rVert  _2$. See, e.g., \cite{Gerard} for some more details about $\dot{H}^s(\R^d)$ and the fractional Laplacian $(-\Delta)^s$. We will always consider $\dot{H}^s(\R^d)$ to be equipped with that norm.  Finally, in \eqref{bianchi-egnell} and henceforth we employ the notation 
\[ \dist(f, \mathcal M) := \inf_{h \in \mathcal M} \lVert \ds (f - h)\rVert  _2 \] 
for the distance in $\dot{H}^s(\R^d)$ between $f$ and $\mathcal M$. (Since the value of $s \in (0, d/2)$ can be considered as fixed throughout, we choose to not include the parameter $s$ into the notation for $\mathcal M$, $S_d$, $2^*$ etc. in order to keep a lighter notation, with the exception of $c_{\mathrm{BE}}(s)$.) 

Our first main result is the strict inequality with respect to the constant coming from two non-interacting bubbles. 

\begin{theorem}
\label{theorem strict ineq 2 bubbles}
For every $d \geq 1$ and $s \in (0, d/2)$, one has $c_{\mathrm{BE}}(s) < 2 - 2^{\frac{d-2s}{d}}$. 
\end{theorem}

The proof of Theorem \ref{theorem strict ineq 2 bubbles} consists in an asymptotic expansion of a sum $u_\lambda$ of two Talenti bubbles supported on length scales $1$ and $\lambda^{-1}$ respectively, as $\lambda \to 0$. It can be verified that $\mathcal E(u_\lambda) \to  2 - 2^{\frac{d-2s}{d}}$ as $\lambda \to 0$, and that the first lower-order correction comes with a negative sign. This approach bears some similarity with that used in \cite{Koenig} to prove the one-bubble inequality $c_{\mathrm{BE}}(s) < \frac{4s}{d+2s+2}$ for $d \geq 2$, but the details of the required arguments and computations are entirely different. We give the proof of Theorem \ref{theorem strict ineq 2 bubbles} in Section \ref{section strict inequality proof} below. 

As a consequence of Theorem \ref{theorem strict ineq 2 bubbles}, together with some further analysis, we obtain that all minimizing sequences for \eqref{bianchi-egnell} must converge towards a non-trivial minimizer.

\begin{theorem}
\label{theorem minimizer}
Let $d \geq 2$ and $s \in (0, \frac{d}{2})$.  
Let $(u_n)$ be a minimizing sequence for \eqref{bianchi-egnell} with $\lVert \ds u_n\rVert  _2 = 1$. Then there is $u \in \dot{H}^s(\R^d) \setminus \mathcal M$ such that, up to extracting a subsequence, $u_n \to u$ strongly in $\dot{H}^s(\R^d)$. Moreover, $\mathcal E(u) = c_{\mathrm{BE}}(s)$, i.e., $u$ is a minimizer for $c_{\mathrm{BE}}(s)$. 
\end{theorem}

Let us give a brief overview over the main ideas of the proof of Theorem \ref{theorem minimizer}. Its basic strategy is similar to existence proofs for simpler functionals,  e.g. the Sobolev functional $\mathcal S(f) = \lVert \ds f\rVert  _2^2 / \lVert f\rVert  _{2^*}^2$, and goes back to Brezis' and Lieb's work \cite{BrLi}. After one has extracted a non-zero weak limit $f$ from a minimizing sequence $(f_n)$, a suitable convexity property of the functional together with the Brezis--Lieb lemma from \cite{BrLi} shows that the value of $\mathcal S(f_n)$ can be strictly improved unless the weak limit has full mass. The sequence $(f_n)$ being minimizing, a strict improvement is excluded, and hence the weak limit has full mass and is in fact a strong limit and a minimizer. 

In the present situation, this idea is harder to put into practice due to the more complicated structure of the Bianchi-Egnell functional \eqref{bianchi-egnell}. The main difficulty stems from the term $\dist(u_n, \mathcal M)^2$, which cannot be split 'symmetrically' under a decomposition $u_n = f + g_n$. This makes it less obvious to deduce a strict improvement of $\mathcal E(u_n)$ using convexity. To overcome this, we first find by some careful arguments (including convexity) that $f$ and $g_n$ must both be rescaled and translated bubbles of equal mass unless $g_n$ vanishes asymptotically. Then we can rule out $f$ and $g_n$ both being bubbles by the strict inequality from Theorem \ref{theorem strict ineq 2 bubbles}. Finally, $u_n$ cannot converge to $\mathcal M$ because of the strict inequality from \cite{Koenig}. (Note that the result from \cite{Koenig} is only valid for $d \geq 2$. This is the only place in the proof where we use this assumption.) We refer to the proof of Theorem \ref{theorem minimizer} and Remark \ref{remark two bubbles} for more details. 

This last step, which uses Theorem \ref{theorem strict ineq 2 bubbles} and \cite{Koenig}, reflects a widely known phenomenon in non-compact minimization problems going back to, at least, Lieb's lemma in the seminal paper by Brezis and Nirenberg \cite[Lemma 1.2]{BrNi}. The decisive observation is that, for many variational problems with some loss of compactness,  minimizing sequences do converge if they lie strictly below some universal energy threshold given by some limit problem. Remarkably, in the present context of the Bianchi-Egnell inequality, we find that  \emph{two} such compactness thresholds (i.e. the fact that $c_{\mathrm{BE}}$ is strictly smaller than both of them) need to be taken into account to prove existence of a minimizer, namely the value $c_{\mathrm{BE}}^\text{2-peak}(s) = 2 - 2^\frac{2}{2^*}$ and the value $c_{\mathrm{BE}}^\text{spec}(s) = \frac{4s}{d+2+2s}$. 

One can note that for this argument, at least when $s =1$, the independent proof of Theorem \ref{theorem strict ineq 2 bubbles} we give below is really only relevant for large dimensions $d$. Indeed, when $\frac{4s}{d+2+2s} < 2 - 2^\frac{d-2s}{d}$ (e.g. when $s = 1$ and $3 \leq d \leq 6$, as already observed above), then Theorem \ref{theorem strict ineq 2 bubbles} follows immediately from the standard spectral inequality $c_{\mathrm{BE}}(s) \leq \frac{4s}{d+2+2s}$. Conversely, when $\frac{4s}{d+2+2s} \geq 2 - 2^\frac{d-2s}{d}$ (e.g. when $s = 1$ and $d \geq 7$), then Theorem \ref{theorem strict ineq 2 bubbles} implies the result from \cite{Koenig}.

\section{Preliminaries}

We start by introducing some more notation. First, we denote the standard $L^{2^*}(\R^d)$-\emph{normalized} Talenti bubble centered at zero by 
\begin{equation}
\label{B definition normalized}
 B(x) = c_d (1 + |x|^2)^{-\frac{d-2s}{2}} 
\end{equation}
with $c_d > 0$ chosen such that $\lVert B\rVert  _{2^*} = 1$. For $\lambda > 0$, $x \in \R^d$, denote 
\[ B_\xl(y) = \lambda^\frac{d-2s}{2} B(\lambda(x-y)).  \]
If $x = 0$, we also write $B_\lambda = B_{0,\lambda}$. Notice that $\lVert B_\xl\rVert  _{2^*} = 1$, $\lVert \ds B_\xl\rVert  _2^2 = S_d$ and $(-\Delta)^s B_\xl = S_d B_\xl^{2^*-1}$ on $\R^d$ for all $x$ and $\lambda$. 

We denote by
\[ \mathcal M_1 = \left\{ B_\xl \, : \, x \in \R^d, \, \lambda > 0 \right\} \subset \mathcal M \]
the submanifold of $\mathcal M$ consisting of normalized Talenti bubbles.

The manifolds $\mathcal M$ and $\mathcal M_1$ are invariant under conformal transformations of $\R^d$, i.e., dilations, translations and inversions. For later reference we collect some explicit transformations in the following lemma, which can be verified by simple computation. 

\begin{lemma}[Conformal transformations of Talenti bubbles]
\label{lemma conformal trafos}
\begin{enumerate}[(i)]
\item Let $D_\mu(x) = \mu x$ be dilation by $\mu > 0$ and set $(D_\mu u)(x) = \mu^\frac{d-2s}{2} u(\mu x)$. Then $D_\mu B_\lambda = B_{\mu \lambda}$. 

\item Let $I_\tau(x) = \frac{\tau^2 x}{|x|^2}$ be the inversion about $\partial B(0, \tau)$ for some $\tau > 0$ and set $(I_\tau u)(x) = \left(\frac{\tau}{|x|}\right)^{d-2s} u\left( \frac{\tau^2 x}{|x|^2} \right)$. Then $I_\tau B_\lambda = B_{\tau^{-2} \lambda^{-1}}$.  
\end{enumerate}
\end{lemma}

The next lemma gives a convenient reformulation of the distance $\dist(f, \mathcal M)$ in terms of a new optimization problem,
\begin{equation}
\label{m(f) definition}
\mathsf m(f):= \sup_{h \in \mathcal M_1} (f, h^{2^*-1})^2,
\end{equation} 
which can be considered as simpler since it is over the smaller set $\mathcal M_1$ and involves no derivative. Here, $(f, h^{2^*-1}) = \int_{\R^d} f h^{2^*-1} \diff x$ denotes the pairing between $L^{2^*}$ and its dual $(L^{2^*})'$. (Be aware that $\mathsf m(f)$ thus defined has nothing in common with the quantity denoted $\mathsf m(\nu)$ that appears in \cite{DoEsFiFrLo}.) We will mostly work with this reformulation when proving our main results below. 

\begin{lemma}
\label{lemma distance}
Let $f \in \dot{H}^s(\R^d)$. Then
\begin{equation}
\label{dist identity lemma}
\dist(f, \mathcal M)^2 = \lVert \ds f\rVert  _2^2 - S_d \,  \mathsf m(f). 
\end{equation} 
 
Moreover, $\dist(f, \mathcal M)$ is achieved. The function $(f, h^{2^*-1}) h$ optimizes $\dist(u, \mathcal M)$ if and only if $h \in \mathcal M_1$ optimizes $\mathsf m(f)$. 
\end{lemma}

For $s=1$, identity \eqref{dist identity lemma} is precisely the one given in \cite[Lemma 3]{DoEsFiFrLo}. Its simple proof readily extends to all $s \in (0, \frac{d}{2})$.

\begin{proof}
Recall that $\Ds h = S_d h^{2^* - 1}$. For any $c \in \R$ and $h \in \mathcal M_1$, we can thus write 
\begin{align*}
 \lVert  \ds (f - ch)\rVert  _2^2 &= \lVert \ds f\rVert  _2^2 - 2 c S_d (f, h^{2^*-1}) + c^2 S_d  \\
 &= \lVert \ds f\rVert  _2^2 - S_d (f, h^{2^*-1})^2 + S_d \left(c - (f, h^{2^* - 1})\right)^2
\end{align*}
by completing a square. Hence 
\begin{align*}
\dist(f, \mathcal M)^2 = \inf_{h \in \mathcal M_1} \inf_{c \in \R}  \lVert  \ds (f - ch)\rVert  _2^2 = \lVert \ds f\rVert  _2^2  - S_d \sup_{h \in \mathcal M_1} (f, h^{2^*-1})^2
\end{align*}
as claimed. The relation between the optimizers of $\dist(u, \mathcal M)$ and $\mathsf m(f)$ is now immediate from the fact that $\inf_{c \in \R} (c - (f, h^{2^* - 1}))^2$ is attained uniquely at $c = (f, h^{2^* - 1})$. 

By this relation between optimizers, it only remains to prove that $\mathsf m(f)$ is always achieved. Let $(B_{x_n, \lambda_n})_n$ be a minimizing sequence for $\mathsf m(f)$. This sequences converges to some $B_{x,\lambda}$, which plainly is a minimizer, unless $\lambda_n \to 0$, $\lambda_n \to \infty$ or $|x_n| \to \infty$. In all three cases it is easy to see that $(f, B_{x_n, \lambda_n}^{2^*-1}) \to 0$ as $n \to \infty$. So to exclude this case, and therefore establish existence of a minimizer, it is sufficient to show that $\mathsf m(f) > 0$.

We will show more, namely that in fact $\mathsf m(f) > 0$ for any non-zero $f \in L^{2^*}(\R^d) \supset \dot{H}^s(\R^d)$ (for which $\mathsf m(f)$ is still well-defined). Indeed, if $f$ is continuous, then for every $x \in \R^d$ one has $(f,B_{x, \lambda}^{2^*-1}) = c f(x) \lambda^{-\frac{d-2s}{2}} + o(\lambda^{-\frac{d-2s}{2}})$ as $\lambda \to \infty$, where $c >0$ is some dimensional constant. Thus $\mathsf m(f) > 0$ unless $f \equiv 0$ (in which case $\mathsf m(0) = 0$ is trivially achieved). For the general case $f \in L^{2^*}(\R^d)$, consider a sequence of continuous functions such that $f_k \to f$ in $L^{2^*}(\R^d)$ and $f_k \to f$ a.e.. If $f \not \equiv 0$ a.e., we can thus find $x \in \R^d$ such that $f_k(x) \to f(x) \neq 0$. Then $(f, B_{x,\lambda}^{2^*-1}) = (c + o(1)) f(x) \lambda^\frac{d-2s}{2}$ follows by the properties of $f_k$. This completes the proof. 
\end{proof}

The next two elementary lemmas will be needed naturally in the framework of the Brezis--Lieb-type argument we use to prove Theorem \ref{theorem minimizer}, as explained above. 

\begin{lemma}
\label{lemma convexity}
Let $p > 2$. Then the function $g(t) = (1 + t^\frac{p}{2})^\frac{2}{p}$ is strictly convex on $(0, \infty)$. In particular, 
\[\eta \mapsto \frac{(1 + \eta^{p})^\frac{2}{p} - 1}{\eta^2} \]
is strictly increasing in $\eta \in (0, \infty)$. 
\end{lemma}

\begin{proof}
We write $q= \frac{p}{2} > 1$. A computation shows that 
\[ g''(t) = (q-1) t^{q-2} (1+t^q)^{\frac{1}{q}-2}. \]
Hence $g''(t) > 0$ for all $t >0$, i.e. $g$ is strictly convex on $(0, \infty)$. 

Now the function $\frac{(1 + \eta^{p})^\frac{2}{p} - 1}{\eta^2} = \frac{g(\eta^2) - g(0)}{\eta^2}$ is strictly increasing in $\eta$ by strict convexity of $g$. 
\end{proof}

The next lemma describes how the value of a quotient changes when summands of the numerator and denominator change in a certain fashion.

\begin{lemma}
\label{lemma quotient inequalities}
Let $A,B,C,D,E,F > 0$ be such that 
\[ \frac{A}{B} \geq \frac{C}{D} \geq \frac{E}{F}, \qquad \text{ and }  \quad D \leq F. \]
Then 
\[ \frac{A}{B} \geq \frac{A + C}{B + D} \geq  \frac{A + E}{B + F}. \]
Moreover, one has $\frac{A}{B} > \frac{A + C}{B + D}$ strictly if $\frac{A}{B} > \frac{C}{D}$ strictly. Likewise, $\frac{A + C}{B + D} > \frac{A + E}{B + F}$ if either $\frac{C}{D} > \frac{E}{F}$ or $D < F$. 
\end{lemma}

\begin{proof}
Using that $\frac{DE}{F} \leq C$ by assumption, we write
\[ \frac{A+ E}{B+ F}  = \frac{A + \frac{F}{D} \frac{DE}{F}}{B + \frac{F}{D}D} \leq \frac{A + \frac{F}{D} C}{B + \frac{F}{D} D}. \]
The inequality $\frac{A}{B} \geq \frac{C}{D}$ entails that the function $t \mapsto \frac{A + tC}{B + tD}$ is decreasing on $(0, \infty)$. Since $\frac{F}{D}  \geq 1$, this yields the conclusion (with non-strict inequalities). A simple inspection of the above proof also yields all the claims about the strict inequalities. 
\end{proof}

\section{Proof of Theorem \ref{theorem strict ineq 2 bubbles}}
\label{section strict inequality proof}

In this section we prove Theorem \ref{theorem strict ineq 2 bubbles}. We will do so by considering a sequence of test functions of the form 
\begin{equation}
\label{u lambda definition}
u_\lambda(x) := B(x) + B_\lambda(x), 
\end{equation}
as $\lambda \to 0$. Recall that the normalized Talenti bubble $B$ is given by \eqref{B definition normalized}, and that $B_\lambda(x) = B_{0,\lambda}(x) = \lambda^\frac{d-2s}{2} B(\lambda x)$. 

The following proposition contains the needed expansion of the terms appearing in $\mathcal E(u_\lambda)$. 

\begin{proposition}
\label{proposition expansion u lambda}
Let $c_0 := B(0) \int_{\R^d} B^\frac{d+2s}{d-2s} \diff x$. As $\lambda \to 0$, the following holds. 
\begin{enumerate}[(i)]
\item \label{expansion gradient} $\int_{\R^d} |\ds u_\lambda|^2 \diff x = 2 S_d + 2 S_d c_0 \lambda^\frac{d-2s}{2} + o(\lambda^\frac{d-2s}{2})$. 
\item \label{expansion p norm} $\lVert u_\lambda\rVert  _{2^*}^2 = 2^\frac{2}{2^*} + 2^{\frac{2}{2^*} +1} c_0 \lambda^\frac{d-2s}{2} + o (\lambda^\frac{d-2s}{2})$. 
\item \label{expansion dist} $\dist(u_\lambda, \mathcal M)^2 = S_d + o(\lambda^\frac{d-2s}{2})$.
\end{enumerate}
\end{proposition}

Using these expansions, the proof of our first main result is immediate. 

\begin{proof}
[Proof of Theorem \ref{theorem strict ineq 2 bubbles}]
By Proposition \ref{proposition expansion u lambda}, as $\lambda \to 0$, we have
\begin{align*}
\mathcal E(u_\lambda) &= \frac{(2 - 2^\frac{2}{2^*}) S_d - S_d (2^{\frac{2}{2^*} + 1} - 2) c_0 \lambda^\frac{d-2s}{2}}{S_d} + o(\lambda^\frac{d-2s}{2}) \\
&= 2 - 2^\frac{2}{2*} - (2^{\frac{2}{2^*} + 1} - 2) c_0 \lambda^\frac{d-2s}{2} + o(\lambda^\frac{d-2s}{2}) < 2 - 2^\frac{2}{2*}
\end{align*}
for $\lambda > 0$ small enough, which is what we wanted to prove.
\end{proof}

It remains to give the proof of Proposition \ref{proposition expansion u lambda}. 

\begin{proof}
[Proof of Proposition \ref{proposition expansion u lambda}]
Let us first prove \eqref{expansion gradient}. Clearly, 
\[ \lVert \ds u_\lambda\rVert  _2^2  = \lVert \ds B\rVert  _2^2 + \lVert \ds B_\lambda\rVert  _2^2 + 2 \langle B, B_\lambda \rangle_{\dot{H}^s(\R^d)} = 2 S_d + 2 \langle B, B_\lambda \rangle_{\dot{H}^s(\R^d)}. \]
Now integrating by parts and using the equation $\Ds B = S_d B^\frac{d+2s}{d-2s}$, we find 
\begin{align*}
\langle B, B_\lambda \rangle_{\dot{H}^s(\R^d)} = S_d \int_{\R^d} B^\frac{d+2s}{d-2s} B_\lambda \diff x = S_d c_0 \lambda^\frac{d-2s}{2} + o(\lambda^\frac{d-2s}{2}).
\end{align*}
Next, let us prove \eqref{expansion p norm}.  Using that inversion about $\partial B(0, \lambda^{-1/2})$ transforms $B + B_\lambda$ into itself by Lemma \ref{lemma conformal trafos}, we can write 
\[ \int_{\R^d} (B + B_\lambda)^{2^*} \diff x = 2 \int_{B(0, \lambda^{-1/2})}   (B + B_\lambda)^{2^*} \diff x. \]
On $B(0, \lambda^{-1/2})$, we have $0 \leq B_\lambda \leq B$. Since $(1 + a)^{2^*} = 1+ a 2^* + \mathcal O(a^2)$ uniformly for $a \in [0,1]$, we have  
\begin{align*}
(B + B_\lambda)^{2^*} &= B^{2^*} \left( 1+ \frac{B_\lambda}{B}\right)^{2^*} = B^{2^*} \left( 1+ 2^* \frac{B_\lambda}{B} + \mathcal O\left(\frac{B_\lambda^2}{B^2}\right)\right) \\
& = B^{2^*} + 2^* B^{2^* - 1} B_\lambda + \mathcal O( B_\lambda^2 B^{2^* - 2}) 
\end{align*} 
uniformly on $B(0, \lambda^{-1/2})$. Now straightforward calculations show 
\[ \int_{B(0, \lambda^{-1/2})} B^{2^*} \diff x = 1 + \mathcal O(\lambda^\frac{d}{2}) = 1 + o (\lambda^\frac{d-2s}{2}), \]
\[ \int_{B(0, \lambda^{-1/2})} B_\lambda B^{2^* - 1} \diff x = \lambda^\frac{d-2s}{2} c_0 + o(\lambda^\frac{d-2s}{2}) \]
by dominated convergence, and 
\[ \int_{B(0, \lambda^{-1/2})} B_\lambda^2 B^{2^* - 2} \diff x = \mathcal O(\lambda^\frac{d}{2}) = o(\lambda^\frac{d-2s}{2}). \]
Thus, in summary 
\[ \int_{\R^d} (B + B_\lambda)^{2^*} \diff x = 2 + 2 \times 2^* c_0 \lambda^\frac{d-2s}{2} + o(\lambda^\frac{d-2s}{2}).  \]
Now \eqref{expansion p norm} follows from a first-order Taylor expansion of $t \mapsto t^\frac{2}{2^*}$ at $t = 2$. 

We now turn to the proof of \eqref{expansion dist}, which is the most involved. To start with, by Lemma \ref{lemma distance} we can write 
\begin{equation}
\label{dist proof}
\dist (u_\lambda, \mathcal M)^2 = \lVert \ds u_\lambda\rVert  _2^2 - S_d \sup_{h \in \mathcal M_1} (u_\lambda, h^{2^* - 1})^2.
\end{equation} 
Since $u_\lambda$ is positive and radially symmetric-decreasing, $ \sup_{h \in \mathcal M_1} (u_\lambda, h^{2^* - 1})^2$ can be found by optimizing over positive symmetric-decreasing functions in $\mathcal M_1$ only, i.e. 
\[ \sup_{h \in \mathcal M_1} (u_\lambda, h^{2^* - 1})^2 = \sup_{\mu > 0} (u_\lambda, B_\mu^{2^* - 1})^2, \]
where $B_\mu(x) = \mu^\frac{d-2s}{2} B(\mu x)$. In other words we only need to find the maximum of the function of one variable $\mu \in (0, \infty)$ given by
\[ H_\lambda(\mu) := (u_\lambda, B_\mu^{2^* - 1}) = F(\mu) + G_\lambda(\mu) \]
with 
\begin{equation}
\label{F, G definition}
 F(\mu) := (B, B_\mu^{2^* - 1}) \qquad \text { and } \qquad G_\lambda(\mu) := (B_\lambda, B_\mu^{2^*-1}).
\end{equation}
Lemma \ref{lemma conformal trafos}.(ii) with $\tau = \lambda^{-1/2}$ implies $H_\lambda(\mu) = H_\lambda(\mu^{-1} \lambda)$. So we only need to optimize over $\mu \geq \lambda^{1/2}$. 

Using the estimates 
\begin{equation}
\label{F, G estimates}
F(\mu)  \lesssim \min\{ \mu^\frac{d-2s}{2}, \mu^{-\frac{d-2s}{2}} \}, \qquad G_\lambda(\mu) \lesssim \left( \frac{\lambda}{\mu}\right)^\frac{d-2s}{2},  \qquad \text{ uniformly for all } \mu \geq \lambda^\frac{1}{2}, 
\end{equation}
we clearly have $\lim_{\mu \to \infty} H_\lambda(\mu) = 0$. Hence $\sup_{\mu \in [\lambda^{1/2}, \infty)} H_\lambda(\mu)$ is attained at some $\mu(\lambda)$. Moreover, since $H_\lambda(1) > F(1) = 1$, in view of \eqref{F, G estimates} we must have $\mu(\lambda) \in [C^{-1}, C]$ for some $C$. Thus there is $\mu_0 > 0$ such that $\mu(\lambda) \to \mu_0$ as $\lambda \to 0$. 

Again by \eqref{F, G estimates}, we see that 
\[ 1 < H_\lambda(1) \leq H_\lambda(\mu(\lambda)) = F(\mu_0) + o(1). \]
Passing to the limit $\lambda \to 0$ yields $1 = F(\mu_0) = (B, B_{\mu_0}^{2^*-1})$ and thus $\mu_0 = 1$ by the equality condition in Hölder's inequality. 

We have thus proved that $\mu(\lambda) \to 1$. By a Taylor expansion at 1, and since $F'(1)  =0$ by Lemma \ref{lemma F and G},  $\mu(\lambda)$ satisfies
\[ 0 = H'_\lambda(\mu(\lambda)) = F'_\lambda(\mu(\lambda)) + G'_\lambda(\mu(\lambda)) = (F''_\lambda(1)+o(1)) (\mu_\lambda-1) + (1+ o(1)) G'_\lambda(1).  \]
Still by Lemma \ref{lemma F and G}, we have  $F''(1) \neq 0$ and $G'_\lambda(1) \lesssim \lambda^\frac{d-2s}{2}$. Therefore
\begin{equation}
\label{mu - 1 estimate}
|\mu_\lambda - 1| = (1 + o(1)) \frac{|G'_\lambda(1)|}{|F''_\lambda(1) + o(1)|} \lesssim \lambda^\frac{d-2s}{2}.
\end{equation} 
Now we can conclude by inserting the estimate \eqref{mu - 1 estimate} back into $H_\lambda(\mu)$. We obtain, using again Lemma \ref{lemma F and G},
\begin{align*}
\mathsf m(u_\lambda){1/2} &= H_\lambda(\mu(\lambda)) = F(\mu(\lambda)) + G_\lambda(\mu(\lambda)) \\
&= F(1) + F'(1) (\mu(\lambda) - 1) + o(|\mu(\lambda) - 1|) + G_\lambda(1)(1 + o(1))   \\
&= 1 + c_0 \lambda^{\frac{d-2s}{2}} + o(\lambda^\frac{d-2s}{2}). 
\end{align*} 
As a consequence, by  \eqref{dist proof} and the already established part (i) of the proposition, we obtain 
\begin{align*}
\dist(u_\lambda, \mathcal M)^2 &= \lVert \ds u_\lambda\rVert  _2^2 - S_d \mathsf m (u_\lambda) \\
& = 2 S_d + 2 S_d c_0 \lambda^\frac{d-2s}{2} - S_d ( 1 + c_0 \lambda^\frac{d-2s}{2})^2 + o(\lambda^\frac{d-2s}{2}) = S_d + o(\lambda^\frac{d-2s}{2}).  
\end{align*} 
This completes the proof of the proposition. 
 \end{proof}

 \section{Proof of Theorem \ref{theorem minimizer}}
 
In this section we give the proof of Theorem \ref{theorem minimizer}. We let $(u_n)$ be a normalized minimizing sequence for $c_{\mathrm{BE}}(s)$, i.e. 
\begin{equation}
\label{minimizing seq}
\mathcal E(u_n) = c_{\mathrm{BE}}(s) + o(1) \qquad \text{ as } n \to \infty, \qquad \qquad \lVert u_n\rVert  _{2^*} = 1.
\end{equation} 
Then 
\[ \lVert \ds u_n\rVert  _2^2 = (c_{\mathrm{BE}}(s) + o(1)) \dist(u_n, \mathcal M)^2 + S_d \leq (c_{\mathrm{BE}}(s) + o(1)) \lVert \ds u_n\rVert  _2^2 + S_d. \]
Since $c_{\mathrm{BE}}(s) \leq \frac{4s}{d+2s+2} < 1$ by \cite{ChFrWe}, this implies that $(u_n)$ is bounded in $\dot{H}^s(\R^d)$. 
By a theorem of Lions \cite{Lions} (see also \cite{Gerard}), up to translating and rescaling the sequence $(u_n)$, we may assume that $u_n \rightharpoonup f$ weakly in $\dot{H}^s(\R^d)$, for some non-zero $f$.  Letting $g_n := u_n - f$, we can thus write  
\begin{equation}
\label{u = f + g}
u_n = f + g_n,\qquad \text{ for some } f \in \dot{H}^s(\R^d) \setminus \{0\}, \quad g_n \rightharpoonup 0 \, \text{ in } \dot{H}^s(\R^d). 
\end{equation}

We first check that if the convergences are strong, then a minimizer of $c_{\mathrm{BE}}(s)$ must exist. 

\begin{proposition}
\label{proposition minimizer exists}
Let $(u_n)$ satisfy \eqref{minimizing seq} and \eqref{u = f + g}, and suppose that $g_n \to 0$ strongly in $\dot{H}^s(\R^d)$. Then $f$ is a minimizer for \eqref{bianchi-egnell}. 
\end{proposition}

\begin{proof}
If $u_n \to f$ strongly in $\dot{H}^s(\R^d)$, then it is clear that $\lVert \ds u_n\rVert  _2^2 \to \lVert \ds f\rVert  _2^2$ and $\dist(u_n, \mathcal M) \to \dist (f, \mathcal M)$. By Sobolev embedding, we also have $\lVert u_n\rVert  _{2^*} \to \lVert f\rVert  _{2^*}$. Thus  $\mathcal E(u_n) \to \mathcal E(f)$ and $f$ is a minimizer, provided that $\dist(f, \mathcal M) \neq 0$, i.e., that $f \notin \mathcal M$. 

But for sequences $(u_n)$ such that $\dist(u_n, \mathcal M) \to 0$, it is known, e.g. from \cite[Proposition 2]{ChFrWe}, that $\mathcal E(u_n) \geq \frac{4s}{d + 2s +2}$. On the other hand, the result in \cite{Koenig} guarantees that $\lim_{n \to \infty} \mathcal E(u_n) = c_{\mathrm{BE}}(s)  < \frac{4s}{d + 2s + 2}$. Hence the minimizing sequence $(u_n)$ cannot satisfy $\dist(u_n, \mathcal M) \to 0$. As explained above, this finishes the proof.
\end{proof}

The proof of Theorem \ref{theorem minimizer} now consists in showing that $g_n \to 0$ must in fact be the case. 

To do so, let us investigate how the components of $\mathcal E(u_n)$ behave under the decomposition \eqref{u = f + g}.  It is standard to check that the weak convergence implies 
\begin{equation}
\label{decomp gradient}
\lVert \ds u_n\rVert  _2^2 = \lVert \ds f\rVert  _2^2 + \lVert \ds g_n\rVert  _2^2 + o(1), 
\end{equation}
and that, using compact Sobolev embeddings  and the Brezis--Lieb lemma \cite{BrLi}, 
\begin{equation}
\label{decomp p norm}
 \int_{\R^d} |u_n|^{2^*} \diff x = \int_{\R^d} |f|^{2^*} \diff x + \int_{\R^d} |g_n|^{2^*} \diff x + o(1)
\end{equation}
along a subsequence, as $n \to \infty$. Finally, the following lemma gives the important information how the distance $\dist(u_n, \mathcal M)$ decomposes. Recall that by definition $\mathsf m(u) = \sup_{h \in \mathcal M_1} (u, h^{2^*-1})^2$. 

\begin{lemma}
\label{lemma max m}
Let $u_n$ satisfy \eqref{u = f + g}. 
As $n \to \infty$, we have 
\begin{equation}
\label{decomp dist lemma}
\mathsf m(u_n) =  \max \left\{ \mathsf m(f),  \mathsf m(g_n) \right\} + o(1). 
\end{equation} 
In particular, 
\begin{align}
\label{dist = lemma}
 \dist (u_n, \mathcal M)^2 &= \lVert \ds f\rVert  _2^2 + \lVert \ds g_n\rVert  _2^2   - S_d  \max \left\{\mathsf m(f), \mathsf m(g_n) \right\} + o(1) .
\end{align}
\end{lemma}

\begin{proof}
By Lemma \ref{lemma distance}, $\mathsf m({u_n})$ has an optimizer $h_n \in \mathcal M_1$. 
We write $h_n(x) = \mu_n^\frac{d-2s}{2} B(\mu_n(x- x_n))$ and consider two different cases. 

Suppose first that $x_n$ is bounded and $\mu_n$ is bounded away from zero and infinity. Then up to a subsequence $x_n \to x_\infty \in \R^d$ and $\mu_n \to \mu_\infty \in (0,\infty)$, and consequently $h_n^{2^*-1} \to B_{x_\infty, \mu_\infty}^{2^*-1}$ strongly in $L^{(2^*)'}$. But this implies $(g_n, h_n^{2^*-1}) \to 0$ by weak convergence $g_n \rightharpoonup 0$. Thus 
\begin{equation}
\label{dist lemma leq m(f)}
 \mathsf m(u_n) = \left( (f, h_n^{2^*-1}) + (g_n, h_n^{2^*-1})\right)^2 = (f, h_n^{2^*-1})^2 + o(1) \leq \mathsf m(f) + o(1). 
\end{equation}
In the remaining, second case, we have $\mu_n \to 0$, $\mu_n \to \infty$ or $|x_n| \to \infty$ along a subsequence. This can be easily checked to yield $h_n^{2^*-1} \rightharpoonup 0$ in $L^{(2^*)'}$. Thus $(f, h_n^{2^*-1}) \to 0$ in that case, and we get 
\begin{equation}
\label{dist lemma leq m(g)}
 \mathsf m(u_n) = \left( (f, h_n^{2^*-1}) + (g_n, h_n^{2^*-1})\right)^2 = (g_n, h_n^{2^*-1})^2 + o(1) \leq \mathsf m(g_n) + o(1). 
\end{equation}
Combining \eqref{dist lemma leq m(f)} and \eqref{dist lemma leq m(g)}, we get 
\begin{equation}
\label{dist lemma leq max}
\mathsf m(u_n) \leq  \max \left\{ \mathsf m(f),  \mathsf m(g_n) \right\} + o(1), 
\end{equation} 
at least along some subsequence. But our argument shows that from any subsequence we can extract another subsequence such that the inequality \eqref{dist lemma leq max} holds. Thus \eqref{dist lemma leq max} must in fact hold for the entire sequence $(u_n)$. 

In order to establish \eqref{decomp dist lemma}, we will now prove the converse inequality by a similar argument. Let $h_f$ be the optimizer for $\mathsf m(f)$. Then $(g_n, h_f^{2^*-1}) \to 0$ by weak convergence $g_n \rightharpoonup 0$ and thus 
\begin{equation}
\label{dist lemma geq m(f)}
\mathsf m(u_n) \geq (u_n, h_f^{2^*-1})^2 = \left((f, h_f^{2^*-1}) + (g_n, h_f^{2^*-1})\right)^2 = \mathsf m(f) + o(1). 
\end{equation} 
Now, let $h_{g_n}$ be the optimizer for $\mathsf m(g_n)$. We write again $h_{g_n} = \mu_n^\frac{d-2s}{2} B(\mu_n(x - x_n))$ and consider two cases. Suppose first that $\mu_n \to 0$, $\mu_n \to \infty$ or $|x_n| \to \infty$ along a subsequence. Then, as above, $h_{g_n}^{2^*-1} \rightharpoonup 0$ in $L^{(2^*)'}(\R^d)$ and thus $(f, h_{g_n}^{2^*-1}) \to 0$. We obtain in that case
\begin{equation}
\label{dist lemma geq m(g)}
\mathsf m(u_n) \geq (u_n, h_{g_n}^{2^*-1})^2 = \left((f, h_{g_n}^{2^*-1}) + (g_n, h_{g_n}^{2^*-1})\right)^2 = \mathsf m(g_n) + o(1). 
\end{equation} 
If, on the other hand, $\mu_n$, $\mu_n^{-1}$ and $|x_n|$ are bounded, then up to a subsequence $\mu_n \to \mu_\infty \in (0, \infty)$ and $x_n \to x_\infty \in \R^d$. But then $\mathsf m(g_n) = (g_n, h_{g_n}^{2^*-1}) \to 0$ by weak convergence $g_n \rightharpoonup 0$, and so \eqref{dist lemma geq m(g)} holds trivially. 

By the same remark as in the first part of the proof, \eqref{dist lemma geq m(g)} holds in fact along the whole sequence $(u_n)$. Now by combining \eqref{dist lemma geq m(f)} and \eqref{dist lemma geq m(g)} with \eqref{dist lemma leq max}, inequality \eqref{decomp dist lemma} follows. 

Finally, \eqref{dist = lemma} is immediate from Lemma \ref{lemma distance} together with \eqref{decomp gradient} and \eqref{decomp dist lemma}. 
\end{proof}

The next lemma serves as an important preparation for our main argument. Contrary to \eqref{decomp gradient}, \eqref{decomp p norm} and \eqref{decomp dist lemma}, here the minimizing property of $(u_n)$ comes into play.

\begin{lemma}
\label{lemma m(f)=m(g)}
Let $(u_n)$ satisfy \eqref{minimizing seq} and \eqref{u = f + g}. Suppose that there is $c > 0$ such that $\lVert \ds g_n\rVert  _2^2 \geq c$.  Then $\mathsf m(f) = \mathsf m(g_n) + o(1)$ as $n \to \infty$. 
\end{lemma}

\begin{proof}
Assume that, up to extracting a subsequence, 
\begin{equation}
\label{m(g) > m(f)}
\lim_{n \to \infty} \mathsf m(g_n) > \mathsf m(f) 
\end{equation} 
strictly. Multiplying by a constant, we may equivalently consider $\tilde{u}_n = \frac{f}{\lVert g_n\rVert  _{2^*}} + \frac{g_n}{\lVert g_n\rVert  _{2^*}} =: \tilde{f}_n + \tilde{g}_n$ with $\lVert \tilde{g}_n\rVert  _{2^*} = 1$. Notice that $\lVert g_n\rVert  _{2^*}$ is bounded away from zero and $\infty$, hence $\lVert \tilde{f}_n\rVert  _{2^*}$ is. Then, by \eqref{decomp gradient}, \eqref{decomp p norm} and Lemma \ref{lemma max m},
\begin{align*}
\mathcal E(\tilde{u}_n) &= \frac{\lVert \ds \tilde{g}_n\rVert  _2^2 - S_d + \lVert \ds \tilde{f}_n\rVert  _2^2 - S_d \left( \left(1 + \int_{\R^d} |\tilde{f}_n|^{2^*} \diff x  \right)^{\frac{2}{2^*}} - 1  \right) }{\lVert  \ds \tilde{g}_n\rVert  _2^2 - S_d \,  \mathsf m(\tilde{g}_n) + \lVert \ds \tilde{f}_n\rVert  _2^2} + o(1). 
\end{align*}
Our goal is now to estimate the quotient using Lemma \ref{lemma quotient inequalities}. Suppose for the moment that $\tilde{g}_n \notin \mathcal M$ for all $n$. Then set  
\begin{align*}
A&:= \lim_{n \to \infty} \lVert \ds \tilde{g}_n\rVert  _2^2 - S_d, \qquad B:= \lim_{n \to \infty} \lVert  \ds \tilde{g}_n\rVert  _2^2 - S_d \, \mathsf m(\tilde{g}_n)^2, \\ 
C&:= \lim_{n \to \infty} \lVert \ds \tilde{f}_n\rVert  _2^2 - S_d \left( \left(1 + \int_{\R^d} |\tilde{f}_n|^{2^*} \diff x \right)^{2/2^*}  - 1  \right), \qquad D:= \lim_{n \to \infty} \lVert \ds \tilde{f}_n\rVert  _2^2. 
\end{align*}
Notice that $A, B, C, D > 0$ because we assume $\tilde{g}_n \notin \mathcal M$ and because $\lVert \ds \tilde{f}_n\rVert  _2^2$ is bounded away from zero.  Since $ c_{\mathrm{BE}}(s) = \lim_{n \to \infty} \mathcal E(\tilde{u}_n) = \frac{A + C}{B + D}$ and $\frac{A}{B} =\lim_{n \to \infty} \mathcal E(g_n) \geq c_{\mathrm{BE}}(s)$, we must have $\frac{C}{D} \leq c_{\mathrm{BE}}(s)$. 

Now let $F_n$ be the scalar multiple of $\tilde{f}_n$ such that $\mathsf m(F_n) = \mathsf m(\tilde{g}_n)$. Then, as a consequence of \eqref{m(g) > m(f)}, $\lim_{n \to \infty} \lVert F_n\rVert  _{2^*} > \lim_{n \to \infty} \lVert \tilde{f}_n\rVert  _{2^*}$ strictly. By Lemma \ref{lemma convexity}, the function $\eta \mapsto \frac{(1 + \eta^{2^*})^{\frac{2}{2^*}}-1}{\eta^2}$ is strictly increasing, so that
\[ \frac{C}{D} = 1 - \lim_{n \to \infty} \frac{S_d \left( \left(1 + \lVert \tilde{f}_n\rVert  _{2^*}^{2^*} \right)^{\frac{2}{2^*}} \diff x - 1  \right)}{\mathcal S[\tilde{f}_n] \lVert \tilde{f}_n\rVert  _{2^*}^2 } > 1 - \lim_{n \to \infty} \frac{S_d \left( \left(1 + \lVert F_n\rVert  _{2^*}^{2^*} \right)^{\frac{2}{2^*}} \diff x - 1  \right)}{\mathcal S[\tilde{f}_n] \lVert F_n\rVert  _{2^*}^2 } =: \frac{E}{F}. \]
Since $D \leq F$ and $\frac{A}{B} \geq \frac{C}{D} > \frac{E}{F}$, Lemma \ref{lemma quotient inequalities} yields 
\begin{align*}
c_{\mathrm{BE}}(s)= \lim_{n \to \infty} \mathcal E(u_n)= \lim_{n \to \infty} \mathcal E(\tilde{u}_n) = \frac{A + C}{B + D} > \frac{A + E}{B + F} = \lim_{n \to \infty} \mathcal E(F_n + \tilde{g}_n). 
\end{align*}
But this contradicts the definition of $c_{\mathrm{BE}}(s)$. Hence \eqref{m(g) > m(f)} is impossible. 

If, on the other hand, $\tilde{g}_n \in \mathcal M$ along some subsequence, then $A = B = 0$ in the above and we directly conclude a contradiction in the same way from $\frac{C}{D} > \frac{E}{F}$. 
 
The remaining case to treat is that where, up to a subsequence, 
\[ \mathsf m(f) > \lim_{n \to \infty} \mathsf m(g_n) . \]
But here one arrives at a contradiction in a similar fashion, with the roles of $f$ and $g_n$ reversed and considering $\tilde{u}_n = \frac{f}{\lVert f\rVert  _{2^*}} + \frac{g_n}{\lVert f\rVert  _{2^*}} =: \tilde{f} + \tilde{g}_n$. The fact that $D:= \lim_{n \to \infty} \lVert \ds g_n\rVert  _2^2 > 0$ is guaranteed here by assumption. The rest of the proof is identical to the above.
\end{proof}

We are now ready to prove our second main result. 

\begin{proof}
[Proof of Theorem \ref{theorem minimizer}]
Let $(u_n)$ be a minimizing sequence satisfying \eqref{minimizing seq} and \eqref{u = f + g}. Suppose for contradiction that $g_n = u_n - f$ does not converge strongly to zero in $\dot{H}^s(\R^d)$. Then, after passing to a subsequence, we have $\lVert \ds g_n\rVert  _2^2 \geq c$ for some $c > 0$. Thus Lemma \ref{lemma m(f)=m(g)} asserts that \begin{equation}
\label{m(f) = m(g) proof thm}
\mathsf m(f) = \mathsf m(g_n) + o(1).
\end{equation}

Suppose first that $\lVert g_n\rVert  _{2^*} \leq \lVert f\rVert  _{2^*} + o(1)$. As in the proof of Lemma \ref{lemma m(f)=m(g)}, we may moreover assume that $\lVert f\rVert  _{2^*} = 1$ by multiplying with a suitable scalar factor. Due to \eqref{m(f) = m(g) proof thm} and Lemma \ref{lemma max m} we may write 
\[ \dist(u_n, \mathcal M)^2 = \lVert \ds f\rVert  _2^2 + \lVert \ds g_n\rVert  _2^2 - S_d \, \mathsf m(f) + o(1).  \]
Together with \eqref{expansion gradient} and \eqref{expansion p norm}, we thus obtain
\begin{align*}
c_{\mathrm{BE}}(s) +o (1) = \frac{\lVert \ds f\rVert  _2^2 - S_d + \lVert \ds g_n\rVert  _2^2 - S_d \left( \left( 1 + \lVert g_n\rVert  _{2^*}^{2^*} \right)^\frac{2}{2^*} - 1 \right)}{\lVert  \ds f\rVert  _2^2 -  S_d \, \mathsf m(f) + \lVert \ds g_n \rVert  _2^2 }
\end{align*}
Similarly to the proof of Lemma \ref{lemma m(f)=m(g)}, since by \eqref{bianchi-egnell}
\[ \frac{\lVert \ds f\rVert  _2^2 - S_d}{\lVert  \ds f\rVert  _2^2 -  S_d \, \mathsf m(f)} \geq c_{\mathrm{BE}}(s), \]
and since $\lVert \ds g_n\rVert  _2^2 \geq c$, we must have
\begin{align} c_{\mathrm{BE}}(s) + o(1) &\geq \frac{\lVert \ds g_n\rVert  _2^2 - S_d \left( \left( 1 + \lVert g_n\rVert  _{2^*}^{2^*} \right)^\frac{2}{2^*} - 1 \right)}{\lVert \ds g_n \rVert  _2^2} \label{chain of ineqs proof} \\
& = 1 - \frac{S_d \left( \left( 1 + \lVert g_n\rVert  _{2^*}^{2^*} \right)^\frac{2}{2^*} - 1 \right)}{\mathcal S(g_n) \lVert g_n\rVert  _{2^*}^2}  \geq 1 - \frac{S_d}{\mathcal S(g_n)} (2^\frac{2}{2^*} - 1), \nonumber
\end{align}
where the last inequality follows from Lemma \ref{lemma convexity} together with $\lVert g_n\rVert  _{2^*} \leq 1$. (Recall that $\mathcal S(g)$ is the Sobolev quotient defined in \eqref{sobolev quotient definition}.) Since we know by Theorem \ref{theorem strict ineq 2 bubbles} that $c_{\mathrm{BE}} < 2 - 2^\frac{2}{2^*}$ with strict inequality, we find, for $n$ large enough, 
that 
\[ 1 - \frac{S_d}{\mathcal S(g_n)} (2^\frac{2}{2^*} - 1) <   2 - 2^\frac{2}{2^*},  \]
which is equivalent to $\mathcal S(g_n) < S_d$. But this contradicts the definition of $S_d$. 

If we assume instead the reverse inequality $\lVert f\rVert  _{2^*} \leq \lVert g_n\rVert  _{2^*} + o(1)$, we obtain a contradiction by writing 
\[ \dist(u_n, \mathcal M)^2 = \lVert \ds f\rVert  _2^2 + \lVert \ds g_n\rVert  _2^2 - S_d \, \mathsf m(g_n) + o(1) \]
and arguing in exactly the same way with the roles of $f$ and $g_n$ reversed. 

Thus we have shown that $g_n$ must converge strongly to zero in $\dot{H}^s(\R^d)$. By Proposition \ref{proposition minimizer exists}, the proof of Theorem \ref{theorem minimizer} is now complete. 
\end{proof}
 
In the previous proof, notice that it is the crucial information \eqref{m(f) = m(g) proof thm} from Lemma \ref{lemma m(f)=m(g)} which allows us to express $\dist(u_n, \mathcal M)^2$ with the help of either $\mathsf m(f)$ or $\mathsf m(g_n)$. This is what permits us to reverse the roles of $f$ and $g_n$, in other words to assume an inequality between $\lVert f\rVert  _{2^*}$ and $\lVert g_n\rVert  _{2^*}$ without loss of generality. 
 
\begin{remark}
\label{remark two bubbles}
The following remark may help to gain some more intuition about the proof of Theorem \ref{theorem minimizer}. If we did not have Theorem \ref{theorem strict ineq 2 bubbles} available, but only the non-strict inequality $c_{\mathrm{BE}} \leq 2 - 2^\frac{2}{2^*}$, then the chain of inequalities \eqref{chain of ineqs proof} would still imply that $\lVert g_n\rVert  _{2^*} \to 1$ and $\mathcal S(g_n) \to S_d$, that is, $\dist(g_n, \mathcal M_1) \to 0$. By $\lVert f\rVert  _{2^*} = 1$, $\mathsf m(f) = \mathsf m(g_n)$ and the equality condition in Hölder's inequality we would then also have $f \in \mathcal M_1$. Thus the (weaker) conclusion would be in this case that either $(u_n)$ converges strongly or, up to rescaling and translation $u_n = B + B_n + o(1)$ for a sequence $(B_n) \subset \mathcal M_1$ interacting weakly with $B$. 
\end{remark}

\appendix

\section{Some computations}

The following lemma was needed in the proof of Proposition \ref{proposition expansion u lambda}. 

\begin{lemma}
\label{lemma F and G}
Let $F$ and $G_\lambda$ be defined by \eqref{F, G definition}. Then 
\[ F'(1) = 0, \quad F''(1) = a_d \]
and 
\[ G_\lambda(1) = c_0 \lambda^\frac{d-2s}{2} + o(\lambda^\frac{d-2s}{2}), \qquad   G'_\lambda(1) = b_d \lambda^\frac{d-2s}{2} + o(\lambda^\frac{d-2s}{2}),\]
for $c_0 = B(0) \int_{\R^d} B^{2^*-1} \diff x$ and
\begin{equation}
\label{a and b}
a_d = -c_d^{2^*} \frac{d-2s}{2(d+1)} \frac{\Gamma (\frac{d+2}{2}) \Gamma (\frac{d}{2})}{\Gamma(d+1)}, \qquad b_d = - c_d^{2^*}  \frac{d-2s}{2(d+2s)} \frac{\Gamma(\frac{d}{2})\Gamma(s)}{\Gamma(\frac{d+2s}{2})}.
\end{equation}
Here $c_d$ is the normalization constant in \eqref{B definition normalized}. 
\end{lemma}

\begin{proof}
We compute 
\[
F'(\mu) = \int_0^\infty r^{d-1}  B(r) \left(  B(\mu r)^{2^* - 1} + \frac{2}{d+2s}  B(\mu r)^{2^* -2}  B'(\mu r) \mu r \right) \diff r .
\]
Hence 
\begin{align*}
c_d^{-2^*} F'(1)&= \int_0^\infty r^{d-1} (1 + r^2)^{-d} \diff r - 2 \int_0^\infty r^{d+1} (1 + r^2)^{-d-1}  \diff r \\
&= \frac{1}{2} \frac{\Gamma(\frac{d}{2})^2}{\Gamma(d)} - \frac{\Gamma(\frac{d+2}{2}) \Gamma(\frac{d}{2})}{\Gamma(d+1)} = 0. 
\end{align*} 
Next, we compute 
\begin{align*}
& \quad F''(\mu) = \frac{\diff }{\diff \mu} \left(\mu^{-d} \int_0^\infty r^{d-1} B(\mu^{-1} r) \left( B(r)^\frac{d+2s}{d-2s} + \frac{2}{d-2s} B^\frac{4s}{d-2s}(r) B'(r) r \right) \diff r \right)  \\
&= - \mu^{-d-1} \int_0^\infty r^{d-1} \left(  d   B(\mu^{-1} r) +   B'(\mu^{-1} r) \mu^{-1} r \right)  \left( B(r)^\frac{d+2s}{d-2s} + \frac{2}{d-2s} B^\frac{4s}{d-2s}(r) B'(r) r \right) \diff r. \\
\end{align*}
Therefore 
\begin{align*}
c_d^{-2^*} F''(1) &= - d F'(1) - c_d^{-2^*}   \int_0^\infty r^{d} \left( B^\frac{d+2s}{d-2s} B'(r)  + \frac{2}{d-2s} B^\frac{4s}{d+2s}(r) B'(r)^2 r \right) \diff r \\
&= (d -2s)  \int_0^\infty \frac{r^{d+1}}{(1 + r^2)^{d+1}} \diff r - 2(d - 2s) \frac{r^{d+3}}{(1 + r^2)^{d+2}} \diff r   \\
&= \frac{d-2s}{2} \frac{\Gamma(\frac{d+2}{2}) \Gamma(\frac{d}{2})}{\Gamma(d+1)} - (d - 2s) \frac{\Gamma (\frac{d}{2}+2) \Gamma (\frac{d}{2})}{\Gamma(d+2)} = - \frac{d-2s}{2(d+1)} \frac{\Gamma (\frac{d+2}{2}) \Gamma (\frac{d}{2})}{\Gamma(d+1)}. 
\end{align*}
Now let us turn to $G_\lambda(\mu)$. First, 
\[ G_\lambda(1) = (B_\lambda, B^{2^*-1}) = \lambda^\frac{d-2s}{2} B(0) \int_{\R^d} B^{2^* - 1} \diff x  + o(\lambda^\frac{d-2s}{2}).\] 
Finally, we have 
\[
G'_\lambda(\mu) = \int_0^\infty r^{d-1}  B_\lambda(r) \left(  B(\mu r)^{2^* - 1} + \frac{2}{d+2s}  B(\mu r)^{2^* -2}  B'(\mu r) \mu r \right) \diff r
\]
and so
\begin{align*}
& \quad  c_d^{-2^*} G'_\lambda(1) \\
 &= c_d^{-2^*} \lambda^\frac{d-2s}{2} \left( \int_0^\infty r^{d-1} B(\lambda r) B(r)^\frac{d+2s}{d-2s}\diff r + \frac{2}{d-2s} \int_0^\infty r^{d-1}  B(\lambda r) B^\frac{4s}{d-2s}(r) B'(r) r \diff r \right) \\
&= \lambda^\frac{d-2s}{2} (1 + o(1))\left( \int_0^\infty \frac{r^{d-1} }{(1+ r^2)^\frac{d+2s}{2}} \diff r  - 2 \int_0^\infty \frac{r^{d+1}}{(1 + r^2)^\frac{d+2s+2}{2}} \diff r  \right) \\
&= \lambda^\frac{d-2s}{2} (1 + o(1)) \left( \frac 12 \frac{\Gamma(\frac{d}{2})\Gamma(s)}{\Gamma(\frac{d+2s}{2})} - \frac{\Gamma(\frac{d+2}{2})\Gamma(s) }{\Gamma(\frac{d+2s+2}{2})} \right) \\
&=- \frac{d-2s}{2(d+2s)} \frac{\Gamma(\frac{d}{2})\Gamma(s)}{\Gamma(\frac{d+2s}{2})} \lambda^\frac{d-2s}{2} + o(\lambda^\frac{d-2s}{2}).
\end{align*}
This completes the proof.
\end{proof}


\begin{thebibliography}{21}


\bibitem{BaCo} A. Bahri, J.-M. Coron, \textit{On a nonlinear elliptic equation involving the critical Sobolev exponent: the effect of the topology of the domain}. Comm. Pure Appl. Math. \textbf{41} (1988), no. 3, 253--294.

\bibitem{BiEg} G.~Bianchi, H.~Egnell,
\textit{A note on the Sobolev inequality}. 
J. Funct. Anal. \textbf{100}, No. 1, 18-24 (1991).

\bibitem{BiCrHe} C.~Bianchini, G.~Croce, A.~Henrot. \textit{On the quantitative isoperimetric inequality in the
plane}. ESAIM Control Optim. Calc. Var., 23(2) (2017), 517--549. 

\bibitem{BoDoNaSi} M.~Bonforte, J.~Dolbeault, B.~Nazaret, N.~Simonov, \textit{Stability in Gagliardo-Nirenberg-Sobolev inequalities: Flows, regularity and the entropy method}. Preprint, arXiv:2007.03674.


\bibitem{BrNi} H.~Br\'ezis, L.~Nirenberg, \textit{Positive solutions of nonlinear elliptic equations involving critical Sobolev exponents}. Comm. Pure Appl. Math. {\bf 36} (1983), 437--477.

\bibitem{BrLi} H.~Brézis, E.~H.~Lieb, \textit{A relation between pointwise convergence of functions and convergence of functionals}.
Proc. Am. Math. Soc. {\bf 88} (1983), 486--490.

\bibitem{BoGoRoSa} S.~G.~Bobkov, N.~Gozlan, C.~Roberto,P.-M.~Samson, \textit{Bounds on the deficit in the logarithmic Sobolev inequality}. J. Funct. Anal. {\bf 267} (2014), No. 11, 4110-4138.

\bibitem{BrDoSi} G.~Brigati, J.~Dolbeault, N.~Simonov, \textit{Logarithmic Sobolev and interpolation inequalities on the sphere: constructive stability results}. Preprint, arXiv:2211.13180. 

\bibitem{CaFrLi} E.~A.~Carlen, R.~L.~Frank, E.~H.~Lieb. \textit{Stability estimates for the lowest eigenvalue of a
Schrödinger operator}. Geom. Funct. Anal., {\bf 24} (2014), No. 1, 63--84.

\bibitem{ChFrWe} S.~Chen, R.~L.~Frank, T.~Weth, \textit{Remainder terms in the fractional Sobolev inequality}. Indiana Univ. Math. J. \textbf{62} (2013), No. 4, 1381--1397.

\bibitem{ChLuTa}
L.~Chen, G.~Lu, H.~Tang, \textit{Sharp Stability of Log-Sobolev and Moser-Onofri inequalities on the Sphere}. J. Funct. Anal. {\bf 285} (2023), no. 5, 110022.

\bibitem{CiFuMaPr} A.~Cianchi, N.~Fusco, F.~Maggi, A.~Pratelli. \textit{The sharp Sobolev inequality in
quantitative form}. J. Eur. Math. Soc. , {\bf 11} (2009), No. 5, 1105--1139.

\bibitem{DoEs} J.~Dolbeault, M.~J.~Esteban, \textit{Hardy-Littlewood-Sobolev and related inequalities: stability}. The physics and mathematics of Elliott Lieb. The 90th anniversary. Volume I. Berlin: European Mathematical Society (EMS). 247--268 (2022).

\bibitem{DoEsFiFrLo} J.~Dolbeault, M.~J.~Esteban, A.~Figalli, R.~L.~Frank, M.~Loss, \textit{Sharp stability for Sobolev and log-Sobolev inequalities, with optimal dimensional dependence}. Preprint, arXiv:2209.08651

\bibitem{FaInLe} M.~Fathi, E.~Indrei, M.~Ledoux, \textit{Quantitative logarithmic Sobolev inequalities and stability estimates}. Discrete Contin. Dyn. Syst. {\bf 36} (2016), No. 12, 6835-6853.

\bibitem{FiNe} A.~Figalli, R.~Neumayer, \textit{Gradient stability for the Sobolev inequality: the case $p>2$}. J. Eur. Math. Soc. (JEMS) {\bf 21} (2019), No. 2, 319-354.

\bibitem{FiZh} A.~Figalli, Y.~R.-Y.~Zhang, \textit{Sharp gradient stability for the Sobolev inequality}. Duke Math. J. {\bf 171} (2022), No. 12, 2407-2459.

\bibitem{Frank} R.~L.~Frank, \textit{The sharp Sobolev inequality and its stability: An introduction}. Preprint, arXiv:2304.03115. 

\bibitem{FuMaPr} N.~Fusco, F.~Maggi, A.~Pratelli, \textit{The sharp quantitative isoperimetric inequality}. Ann. of Math. (2) {\bf 168} (2008), 941--980.

\bibitem{Gerard} P.~Gérard, \textit{Description du défaut de compacité de l'injection de Sobolev}. ESAIM: Control, Optimisation and Calculus of Variations, Volume 3 (1998), 213--233.


\bibitem{Koenig} T.~König, \textit{On the sharp constant in the Bianchi-Egnell stability inequality}. Bull. Lond. Math. Soc. \textbf{55} (2023), no. 4, 2070--2075. 

\bibitem{Lions} P.~-L.~Lions, \textit{The Concentration-Compactness Principle in the Calculus of Variations. The Limit Case, Part 2}. Rev. Mat. Iberoam. {\bf 1} (1985), no. 2, 45--121.


\bibitem{Rey} O.~Rey, \textit{The role of the Green's function in a non-linear elliptic equation involving the critical Sobolev exponent}. J. Funct. Anal. {\bf 89} (1990), 1--52. 


\end{thebibliography}
\end{document}